\documentclass[11pt,english]{article}
\usepackage{amssymb}
\usepackage{amsfonts,yfonts}
\usepackage[all]{xy}
\usepackage{amsfonts,amscd,amssymb,amsthm}
\usepackage{amsmath,amsxtra,amscd,amsthm,eucal,graphicx,graphics}
\usepackage{tikz}
\usepackage{enumerate}
 \usepackage[T1]{fontenc}
\usepackage[utf8]{inputenc}
\usepackage{babel}
\usepackage[T1]{fontenc}
\usepackage[utf8]{inputenc}

\newcommand{\overbar}[1]{\mkern 1.5mu\overline{\mkern-1.5mu#1\mkern-1.5mu}\mkern 1.5mu}

\theoremstyle{definition}
\newtheorem{theorem}{Theorem}[section]
\newtheorem{remark}[theorem]{Remark}

\newtheorem{lemma}[theorem]{Lemma}
\newtheorem{corollary}[theorem]{Corollary}

 \title{Uniform Length Dominating Sequence Graphs}
    \author{Aysel Erey}
    \date{\small Department of Mathematics\\ Gebze Technical University \\ Kocaeli, TURKEY \\ E-mail: aysel.erey@gtu.edu.tr
     \\[\baselineskip] \today }

\begin{document}

\maketitle

\begin{abstract}
A sequence of vertices $(v_1,\, \dots , \,v_k)$ of a graph $G$ is called a {\it dominating closed neighborhood sequence} if $\{v_1,\, \dots , \,v_k\}$ is a dominating set of $G$ and $N[v_i]\nsubseteq \cup _{j=1}^{i-1} N[v_j]$ for every $i$. A graph $G$ is said to be {\it $k-$uniform} if all dominating closed neighborhood sequences in the graph have equal length $k$.  Bre{\v s}ar et al. \cite{bresar_k_uni} characterized $k$-uniform graphs with $k\leq 3$. In this article we extend their work by giving a complete characterization of all $k$-uniform graphs with $k\geq 4$. 
\end{abstract}


 \textup{2010} \textit{AMS Mathematics Subject Classification}: \textup{05C69}
\section{Introduction}

All graphs considered in this article are finite, simple, loopless and undirected. Given a graph $G$, let $\overbar{G}$ be the complement of $G$, and let $V(G)$ and $E(G)$ be the vertex set and the edge set of $G$ respectively. A vertex $u$ is a {\it neighbor} of vertex $v$ if they are adjacent. The {\it open neighborhood of v}, $N_G(v)$, consists of all neighbors of $v$ in $G$, and the {\it closed neighborhood of v}, $N_G[v]$, is equal to $N_G(v)\cup \{v\}$. We simply write $N(v)$ and $N[v]$ for $N_G(v)$ and $N_G[v]$ respectively when the graph $G$ is clear from the context. A vertex $v$ of $G$ is called an {\it isolated vertex} of $G$ if $N(v)=\emptyset$ and it is called a {\it dominating vertex} of $G$ if $N[v]=V(G)$. Two distinct vertices $u$ and $v$ are called {\it true twins} if $N[u]=N[v]$ and they are called {\it false twins} if $N(u)=N(v)$. For a subset $S$  of vertices of $G$, let $\displaystyle N(S)=\cup_{v\in S}N(v)$ and let $G\setminus S$ denote the subgraph induced by the vertices of $V(G)\setminus S$ (if $S=\{u\}$ is a singleton, we simply write $G\setminus u$). Also, $S$ is called a {\it dominating set} of $G$ if $S\cup N(S)=V(G)$, that is, every vertex in $V(G)\setminus S$ is adjacent to some vertex in $S$.  A {\it total dominating set} of a graph $G$ is a subset of vertices $S$ such that $N(S)=V(G)$, that is, every vertex in $V(G)$ is adjacent to some vertex in $S$. Note that every total dominating set of a graph $G$ is also a dominating set of $G$.

The {\it join} of two graphs $G$ and $H$, denoted by $G\vee H$, has vertex set $V(G\vee H)=V(G)\cup V(H)$ and edge set $E(G\vee H)=E(G)\cup E(H)\cup \{uv \, | \, u\in V(G) \ \text{and } v\in V(H) \}$. Let $\vee_t H$ denote the join of $t$ copies of the graph $H$ and $tH$ denote disjoint union of $t$ copies of $H$. The complete graph, path graph and cycle graph on $n$ vertices are denoted by $K_n,P_n$ and $C_n$ respectively. Lastly, let $K_{p_1,p_2}$ denote the complete bipartite graph with partition sizes $p_1$ and $p_2$, and let $K_{p_1,\dots ,p_t}$ denote the complete multipartite graph with $t$ parts of sizes $p_1,\dots ,p_t$ where $t\geq 2$.

A sequence of $k$ distinct vertices $(v_1,\, \dots , \,v_k)$ in $G$ is said to have {\it length} $k$ and it is called a {\it dominating sequence} (respectively {\it total dominating sequence}) of $G$ if the corresponding set $\{v_1, \dots , v_k\}$ is a dominating set (respectively total dominating set) of $G$. A sequence $(v_1,\, \dots , \,v_k)$ is called a {\it closed neighborhood sequence} (abbreviated to CNS) if, for each $i$ with $2\leq i\leq k$, we have 
\begin{center}$ N[v_i]\nsubseteq \cup _{j=1}^{i-1} N[v_j]$,
\end{center} or equivalently, $N[v_i]\setminus \cup _{j=1}^{i-1} N[v_j]\neq \emptyset$.  A sequence of vertices $(v_1, \dots , v_k)$ of a graph $G$ is called an {\it open neighborhood sequence} (abbreviated to ONS) of $G$ if $N(v_1)\neq \emptyset$ and \begin{center}$N(v_i)\nsubseteq \cup _{j=1}^{i-1} N(v_j)$
\end{center} for each $i$ with $2\leq i \leq k$.

The study of neighborhood sequences was initially motivated by some domination games in \cite{game1,game2,game3}. The total version of the domination game was introduced in \cite{total_game} and recently studied in \cite{irsic, jiang}. Variants of such sequences have connections to the minimum rank problem and the so called zero forcing number of the graph \cite{bresar_zero_forc,linalg}. Such sequences are also called {\it legal sequences} in the context of the domination games. Each vertex in a neighborhood sequence dominates (or totaly dominates) a new vertex which is not dominated (or totally dominated) by any of the preceding vertices.
 
 The lengths of these sequences are also related to some other important graph parameters which have been extensively studied in the literature. For example,
 the minimum length of a dominating closed neighborhood sequence of a graph $G$ is the well known {\it domination number} $\gamma (G)$ of $G$ and the maximum length of a dominating closed neighborhood sequence of $G$ is called the {\it Grundy domination number} $\gamma _{gr}(G)$ of $G$. For a graph $G$ with no isolated vertices, the minimum length of a total dominating sequence of $G$ is called the {\it total domination number} $\gamma _t (G)$ of $G$. Note that every minimum total dominating sequence of a graph $G$ is indeed an open neighborhood sequence of $G$, so $\gamma _t (G)$ is equal to the minimum length of a total dominating ONS. If $G$ has no isolated vertices, the maximum length of a total dominating ONS is known as the {\it Grundy total domination number} $\gamma_{gr}^t(G)$ of $G$ \cite{total}.

  We say that a graph $G$ is a {\it uniform length dominating sequence graph} if all  dominating closed neighborhood sequences have the same length. A graph $G$ is called $k${\it -uniform} if all dominating closed neighborhood sequences of $G$ have equal length $k$. In other words, a graph $G$ is $k$-uniform if and only if $\gamma (G)=\gamma _{gr}(G)=k$. 
Bre{\u s}ar et al.~\cite{bresar_k_uni} gave a characterization of $k$-uniform graphs for $k\in \{1,2,3\}$ (see Theorem~$3.6$ in \cite{bresar_k_uni}). In this article, we complete the characterization of $k$-uniform graphs by finding all $k$-uniform graphs with $k\geq 4$ (Corollary~\ref{main_cor}).

\section{Characterization of $k$-uniform graphs for $k\geq 4$}

Bre{\v s}ar et al. \cite{bresar_k_uni} gave the following characterization of $k$-uniform graphs for $k\leq 3$.

\begin{theorem}\label{bresar_thm}\cite{bresar_k_uni}
If $G$ is a graph, then
\begin{itemize}
\item $G$ is $1$-uniform if and only if $G$ is a complete graph;
\item $G$ is $2$-uniform if and only if its complement ${\overbar G}$  is the disjoint union of one or more complete bipartite graphs;
\item $G$ is $3$-uniform if and only if $G$ is the disjoint union of a $1$-uniform and a $2$-uniform graph.
\end{itemize}
\end{theorem}

Observe that if $G$ is a $3$-uniform graph with no true twins, then $\overbar{G}$ is a friendship graph, that is, the join $K_1\vee tK_2$ for some $t\geq 1$. We use this characterization as the basis step of our induction and  extend it to $k\geq 4$ by showing that every $k$-uniform graph is indeed a disjoint union of $1$-uniform and $2$-uniform graphs. To prove our result we make use of the following two observations.


\begin{lemma}\label{k_uni_remark} Let $G$ be a $k$-uniform graph and $v$ be any vertex of $G$. Then,
\begin{itemize}
\item[(i)] the subgraph $G\setminus N[v]$ is $(k-1)$-uniform;
\item[(ii)] if $G$ has no true twins, then $G\setminus N[v]$ has no true twins.
\end{itemize}

\end{lemma}
\begin{proof}
(i) Let  $(v_1,\dots , v_r)$ be a dominating CNS of the subgraph $G\setminus N[v]$. It is clear that $(v_1,\dots , v_r, v)$ is a dominating CNS of $G$, as $v\notin \cup_{i=1}^rN[v_i]$ and $N_{G\setminus N[v]}[v_{i+1}]\setminus N_{G\setminus N[v]}[v_{i}] \neq \emptyset $ implies that $N_{G}[v_{i+1}]\setminus N_{G}[v_{i}] \neq \emptyset $, for each $i\in \{1,\dots , r-1\}$. Hence $r=k-1$ and the result follows.

(ii) Suppose on the contrary that $G\setminus N[v]$ has two true twins, say $w_1$ and $w_2$. There are no true twins in $G$, so there is a vertex $w'\in N(v)$ such that $w'$ is adjacent to exactly one of $w_1$ and $w_2$. Without loss of generality, suppose that $w'w_1\notin E(G)$ and $w'w_2\in E(G)$. In any graph, every sequence consisting of a single vertex can be extended to a dominating CNS of the graph. So, there exist vertices $w_3, \dots , w_k$  such that $(w_2, w_3,\dots , w_k)$ is a dominating CNS of $G\setminus N[v]$ since it is a $(k-1)$-uniform graph. Note that $N_{G\setminus N[v]}[w_{i+1}]\setminus N_{G\setminus N[v]}[w_{i}] \neq \emptyset $ implies that $N_{G}[w_{i+1}]\setminus N_{G}[w_{i}] \neq \emptyset $, for each $i\in \{2,\dots , k-1\}$. Now, $(w_1,w_2,\dots , w_k,v)$ is a dominating CNS of $G$ of length $k+1$, as $w'\in N_G[w_2]\setminus N_G[w_1]$ and $v\notin \cup_{i=1}^kN[w_i]$. The latter contradicts with $G$ being $k$-uniform. Thus, $G\setminus N[v]$ has no true twins.

\end{proof}

\begin{remark}\label{conn_comp_remark}
Let $G$ be a graph with connected components $G_1,\dots , G_c$. Then, $G$ is $k$-uniform if and only if each $G_i$ is $k_i$-uniform where $k=k_1+\cdots +k_c$ and $k_i\geq 1$.
\end{remark}

\begin{theorem}\label{mymain1} If $G$ is a $k$-uniform graph with $k\geq 3$ and $G$ has no true twins, then $G$ is a disjoint union of $1$-uniform and $2$-uniform graphs.
\end{theorem}

\begin{proof}
We proceed by strong induction on $k$. For $k=3$, the result follows from Theorem~\ref{bresar_thm}. Let $G$ be a $k$-uniform graph with no true twins where $k\geq 4$. Let $v$ be a vertex of $G$. By Lemma~\ref{k_uni_remark}, the subgraph $G\setminus N[v]$ is $(k-1)$-uniform and it has no true twins. It follows that $G\setminus N[v]$ is a disjoint union of $1$-uniform and $2$-uniform graphs by the induction hypothesis. By Theorem~\ref{bresar_thm}, every $2$-uniform graph is either connected or disjoint union of two $1$-uniform graphs. Let $G_1, \dots , G_r$ be the $1$-uniform connected components  and $H_1,\dots , H_t$ be the $2$-uniform connected components of the subgraph  $G\setminus N[v]$, if any. Since $G\setminus N[v]$ has no true twins, $G_i\cong K_1$ for each $i\in \{1,\dots , r\}$ and  $H_i\cong \vee_{t_i} \overbar{K_2}$ for some integer $t_i\geq 2$ for each $i\in \{1,\dots , t\}$ (note that $t_i\geq 2$ as $H_i$ is connected). Note that  $k-1=r+2t$ where $r,t\geq 0$ by Remark~\ref{conn_comp_remark}. In the sequel, let $V(G_i)=\{v_i\}$ for each $i\in \{1,\dots , r\}$,  and $u_i$ and $u_i'$ be two nonadjacent vertices of $H_i$ for each $i\in \{1,\dots , t\}$.

Now let us show that $G$ is disconnected. Suppose on the contrary that $G$ is connected.
Let $A_i=N(v_i)\cap N(v)$ and $B_i=N(V(H_i))\cap N(v)$ for each $i$. Since $G$ is connected, $A_i$ and $B_i$ are nonempty for each $i$. Let us show that $A_1,\dots , A_r, B_1,\dots , B_t$ are mutually disjoint.
\begin{itemize}
\item $A_i\cap A_j=\emptyset$ whenever $i\neq j$:

Without loss of generality, suppose on the contrary that there exists a vertex $w$ in $N(v)$ such that $w$ is adjacent to both $v_1$ and $v_2$. Then, $(v,w,v_3,\dots ,v_r, u_1,u_1',\dots , u_t,u_t')$ is a  dominating sequence of $G$ which has length $r+2t=k-1$. By removing some vertices from this sequence, if necessary, we obtain a subsequence of it which is a dominating CNS of $G$ with length at most $k-1$. The latter  contradicts with $G$ being $k$-uniform.

\item $B_i\cap B_j=\emptyset$ whenever $i\neq j$:

Without loss, suppose on the contrary that there exists a vertex $w$ in $N(v)$ such that $w$ is adjacent to both $u_1$ and $u_2$. Then, $(v,w, u_1',u_2',u_3,u_3',\dots , u_t,u_t', v_1,\dots , v_r)$ is a dominating sequence of $G$ which has length $r+2t=k-1$. As in the previous case, we can obtain a dominating CNS of length at most $k-1$ which is a contradiction.

\item $A_i\cap B_j=\emptyset$ for every $i$ and $j$:

Without loss, suppose on the contrary that there exists a vertex $w$ in $N(v)$ such that $w$ is adjacent to both $v_1$ and $u_1$. In this case we obtain a contradiction again by finding  the dominating sequence $(v,w,v_2,\dots v_r, u_1',u_2,u_2',\dots , u_t,u_t')$ of $G$ with length $k-1$.
\end{itemize}

If $r\geq 1$, we consider removing the vertex $v_1$ from $G$.  The vertex $v_1$ has no neighbor in $A_i$ or $B_j$ for each $i\neq 1$ and $j\in \{1,\dots , t\}$, as $A_i$'s are mutually disjoint and $A_i\cap B_j=\emptyset$ for every $i$ and $j$. Also note that $A_i$'s and $B_j$'s are nonempty. So, there is a walk between every pair of vertices in $G\setminus N[v_1]$ via $v$ which makes the subgraph $G\setminus N[v_1]$ connected. But this is a contradiction because  $G\setminus N[v_1]$ is $(k-1)$-uniform by Lemma~\ref{k_uni_remark} and hence must be disconnected by the induction hypothesis.

If  $r=0$, then $G\setminus N[v]$ has $t$ connected components all of which are $2$-uniform graphs. So, $k-1=2t$ in this case. Now we may assume that $k\geq 5$ since $k$ is odd and $k\geq 4$. First let us show that $u_1$ and $u_1'$ have the same neighbors in $N(v)$. Suppose on the contrary that there is a vertex $w$ in $N(v)$ which is adjacent to exactly one of $u_1$ and $u_1'$. Without loss of generality, assume that $wu_1\in E(G)$ and $wu_1'\notin E(G)$. By Lemma~\ref{k_uni_remark}, the subgraph $G\setminus N[u_1']$ is $(k-1)$-uniform and has no true twins. However, $B_i$'s are mutually disjoint, the vertex $u_1$ has a neighbor in $N(v)$ and $u_1$ is the only vertex of $H_1$ which belongs to $G\setminus N[u_1']$. So, the subgraph $G\setminus N[u_1']$ is connected and the latter contradicts with the induction hypothesis. Now, the vertices $u_1$ and $u_1'$ must have the same neighbors in $N(v)$ and  hence, the vertex $u_1$ is an isolated vertex of $G\setminus N[u_1']$. Therefore, the subgraph $(G\setminus N[u_1'])\setminus u_1$ must be $(k-2)$-uniform by Remark~\ref{conn_comp_remark} and it has no true twins. As $k-2\geq 3$, the subgraph $(G\setminus N[u_1'])\setminus u_1$ is a disjoint union of $1$-uniform and $2$-uniform graphs by the induction hypothesis. This is again a contradiction, as $(G\setminus N[u_1'])\setminus u_1$ is connected.

Thus, $G$ is disconnected and the result follows by induction and Remark~\ref{conn_comp_remark}.

\end{proof}

Let $G'$ be a graph obtained from another graph $G$ by adding a new true twin vertex and $k$ be any positive integer. Then, observe that $G'$ is $k$-uniform if and only if $G$ is $k$-uniform. Thus, we obtain a characterization of all $k$-uniform graphs as an immediate consequence of Theorem~\ref{mymain1}.

\begin{corollary}\label{main_cor} Every $k$-uniform graph is a disjoint union of $1$-uniform and $2$-uniform graphs.
\end{corollary}

\section{Concluding remarks}

Let us consider length uniformity for two types of open neighborhood sequences. Let $G$ be a graph   with no isolated vertices. We say that $G$ is {\it open $k$-uniform} if every dominating ONS of $G$ has length $k$. Also, we call a graph $G$ {\it total $k$-uniform} if every total dominating ONS of $G$ has length $k$. Note that a graph $G$ is total $k$-uniform if and only if $\gamma _t(G)=\gamma_{gr}^t(G)=k$.

No graph has a total dominating sequence of length $1$ by definition, so $\gamma _t(G)\geq 2$ for every graph $G$. In \cite{total}, it was shown that $\gamma _t(G)=2$ if and only if $G$ is a multipartite graph (Theorem~4.4), and if $G$ is a graph with $\gamma _t(G)=3$, then $\gamma _{gr}^t(G)>3$ (Theorem 3.2). Hence, these results immediately give a characterization of total $k$-uniform graphs with $k\leq 3$.

\begin{theorem}\cite{total}\label{total_thm}
\begin{itemize}
\item[(i)] There are no total $k$-uniform graphs with $k\in \{1,3\}$;
\item[(ii)] A graph is total $2$-uniform if and only if it is a complete multipartite graph.
\end{itemize}
\end{theorem}

We also remark that in \cite{bresar_do} it was shown that there are no graphs $G$ with $\gamma _{gr}^t(G)=3$ (Proposition~3.9). So, this gives an alternative proof for the nonexistence of total $3$-uniform graphs. Recently, the authors in \cite{total4} studied  the characterizations of total $k$-uniform graphs with $k\in \{4,6\}$ in some graph classes.

Every total dominating sequence in a graph is also a dominating sequence. So, it is clear that the minimum length of a dominating ONS is at most $\gamma _t(G)$ and the maximum length of a dominating ONS is at least   $\gamma _{gr}^t(G)$. Also, if every dominating ONS has  length $k$ in a graph, then every total dominating ONS has length $k$ too. Hence, every open $k$-uniform graph is indeed a total $k$-uniform graph. Moreover, one can show that a longest dominating ONS must be a total dominating sequence. To see this, consider a longest dominating ONS $(v_1,\dots , v_k)$. Suppose on the contrary that it is not a total dominating sequence. So there is a vertex, without loss of generality, say $v_1$, such that $v_1$ is nonadjacent to every vertex in $\{v_2,\dots , v_k\}$. Since graph has no isolated vertices, $v_1$ has a neighbor $w$. Now $(v_1, \dots ,v_k, w)$ is  a dominating ONS which contradicts with $(v_1,\dots , v_k)$ being longest such sequence. We summarize all of this in the following.

\begin{remark}\label{open_rk} Let $G$ be any graph. Then,
\begin{itemize}
\item[(i)] If $G$ is an open $k$-uniform graph, then $G$ is also total $k$-uniform.
\item[(ii)]The length of a minimum dominating ONS in $G$ is at most $\gamma _t(G)$;
\item[(iii)]  The length of a maximum dominating ONS in $G$ is equal to $\gamma _{gr}^t(G)$.
\end{itemize}
\end{remark}

Note that the minimum lengths of a dominating ONS and a  total dominating ONS in a graph may be different. For example, in a complete graph $G$, the minimum length of a  dominating ONS is $1$ whereas $\gamma_t (G)$ is $2$. In a path graph $P_5$, the minimum length of a  dominating ONS is $2$ whereas $\gamma_t (P_5)$ is $3$.

\begin{corollary}\label{openthm} Let $G$ be a graph, then
\begin{itemize}
\item[(i)] There are no open $k$-uniform graphs with $k\in \{1,3\}$;
\item[(ii)] $G$ is open $2$-uniform if and only if $G$ is a complete multipartite graph $K_{p_{1}, \cdots ,p_t}$ where  $p_i\geq 2$ for each $i=1,\dots, t$.
\end{itemize}
\end{corollary}
\begin{proof}
 Theorem~\ref{total_thm} and Remark~\ref{open_rk}~(i) imply that there are no open $k$-uniform graphs with $k\in \{1,3\}$ and every open $2$-uniform graph is a complete multipartite graph. Observe that for a complete multipartite graph $G=K_{p_{1}, \cdots ,p_t}$, if $p_i=1$ for some $i$, then the graph has a dominating vertex $v$, that is, a vertex $v$ with $N[v]=V(G)$. So $G$ has a dominating ONS of length $1$ and it cannot be open $2$-uniform. Thus, if $G$ is open $2$-uniform then $p_i\geq 2$ for each $i$. Lastly, it is easy to check that if $p_i\geq 2$ for each $i$ then $G$ is open $2$-uniform and the result follows.
\end{proof}

Note that the  characterizations of open $k$-uniform and total $k$-uniform graphs may be different in general. In fact, the structure of open $k$-uniform graphs is supposed to be more restricted as it requires additional conditions.

\section{Acknowledgements}

I would like to thank Didem G{\"o}z{\"u}pek and Martin Milani{\v c} for making me aware of the problem of characterization of $k$-uniform graphs. Also, I would like to thank anonymous referees for pointing out the references and observations in Section~3 which greatly reduced the proof of Corollary~\ref{openthm}.
\vskip0.3in

\bibliographystyle{elsarticle-num}

\end{document}